\documentclass[11pt,a4paper]{amsart}
\usepackage{amssymb,amsfonts,amsmath,mathrsfs}
\usepackage[left=2cm,right=2cm,top=2cm,height=250mm]{geometry}
\usepackage{color}
\usepackage{verbatim}
\newtheorem{proposition}{Proposition}[section]
\newtheorem{theorem}[proposition]{Theorem}

\newtheorem{lemma}[proposition]{Lemma}

\theoremstyle{definition}
\newtheorem{definition}[proposition]{Definition}

\newtheorem{remark}[proposition]{Remark}

\numberwithin{equation}{section}

\def\R{\Bbb R}
\def\Dx{\Delta_x}
\def\Nx{\nabla_x}
\def\Dt{\partial_t}
\def\({\left(}
\def\){\right)}

\def\Cal{\mathcal}
\def\Bbb{\mathbb}
\begin{document}
\title[Strongly damped wave equation]{A note on a strongly damped wave equation with
fast growing nonlinearities}
\author[] {Varga Kalantarov${}^1$, and Sergey Zelik${}^2$}

\begin{abstract} A  strongly damped wave equation including the displacement depending
nonlinear damping term and nonlinear interaction function is considered.
The main aim of the note
is to show that under the standard dissipativity restrictions on the nonlinearities
involved the initial boundary value problem for the considered equation is globally
well-posed in the class
of sufficiently regular solutions and the semigroup generated by the problem possesses
a global attractor in the corresponding phase space.
These results are obtained for the nonlinearities
of an arbitrary polynomial growth and without the assumption that the considered
problem has a global Lyapunov function.
\end{abstract}

\thanks{The work of  Varga Kalantarov was partially supported  by
the Scientific and Research Council of
Turkey, grant no. 112T934. }

\address{${}^1$ Department of mathematics, Ko{\c c} University,
\newline\indent Rumelifeneri Yolu, Sariyer 34450\newline\indent
Sariyer, Istanbul, Turkey}
\email{vkalantarov@ku.edu.tr}
\address{${}^2$ University of Surrey, Department of Mathematics, \newline
Guildford, GU2 7XH, United Kingdom.}

\email{s.zelik@surrey.ac.uk}
\maketitle
\section{Introduction}\label{s0}

In a bounded smooth domain $\Omega\subset\R^3$, we consider the following problem:
\begin{equation}\label{0.1}
\begin{cases}
\Dt^2u+f(u)\Dt u-\gamma\Dt\Dx u-\Dx u+g(u)=h,\ \ u\big|_{\partial\Omega}=0,\\
u\big|_{t=0}=u_0,\ \ \Dt u\big|_{t=0}=u_1.
\end{cases}
\end{equation}
Here $u=u(t,x)$ is unknown function, $\Dx$ is the Laplacian with respect to the variable $x$, $\gamma>0$ is a given dissipation parameter, $f$ and $g$ are given nonlinearities and $h\in L^2(\Omega)$ are given external forces.
We assume throughout of the paper that the nonlinearities $f$ and $g$ satisfy
\begin{equation}\label{0.2}
\begin{cases}
1.\ \  f,g\in C^1(R),\\
2.\ \ -C+\alpha|u|^p\le f(u)\le C_1(1+|u|^p),\\
3.\ \  -C+\alpha|u|^q\le g'(u)\le C(1+|u|^q),
\end{cases}
\end{equation}
where $p,q\ge0$ and $p+q>0$.
\par
Strongly damped wave equations of the form \eqref{0.1}  and similar equations
are of a great current interest, see
\cite{CaCho02,CCD1,DP,GM,Ka86,Khan2010,Khan2012,Ma,PS05,PZ06-1,YS09,We}
and references therein. The most studied is the case with only one
nonlinearity ($f\equiv0$), i.e. the problem of the form
\begin{equation}\label{0.1g}
\begin{cases}
\Dt^2u-\gamma\Dt\Dx u-\Dx u+g(u)=h,\ \ u\big|_{\partial\Omega}=0,\\
u\big|_{t=0}=u_0,\ \ \Dt u\big|_{t=0}=u_1.
\end{cases}
\end{equation}

 Even in this particular case, the equation has a lot of non-trivial and interesting features
 attracting the attention of many mathematicians, see \cite{CaCho02,CCD1,GM,Ka86,Ma,We}
 and references therein. For instance, it has been thought for a long time that, for the case of the solutions belonging to the so-called energy phase space, there is
 a critical growth exponent $q_{max}=4$ for the nonlinearity $g$  and that the properties
 of the solutions in the supercritical case $q>q_{max}$ are {\it principally} different
 from the subcritical case $q<q_{max}$. On the other hand, as has been shown already in \cite{We}
 that the problem
 \eqref{0.1g}, with nonlinear term satisfying just the condition
 $g'(s)\geq -C, \ \forall s\in\R$
  has a global unique solution
  belonging to the more regular phase space
\begin{equation}\label{0.smspace}
\Cal E_1:=[H^2(\Omega)\cap H^1_0(\Omega)]\times H^1_0(\Omega),\ \
\xi_u(t)\in\Cal E_1,\ \ t\ge0.
\end{equation}
Here and below $H^s(\Omega)$ stands for the usual Sobolev space of
distributions whose derivatives up to order $s$ belong to
$L^2(\Omega)$ and $H^s_0(\Omega)$ means the closure of
$C_0^\infty(\Omega)$ in $H^s(\Omega)$.
\par
So,  there is no critical growth exponent for 
the class of smooth solutions $\xi_u:=(u,\Dt u)\in \Cal E_1$. A dissipativity of
the semigroup generated by  problem \eqref{0.1g} in the phase space
$\Cal E_1 $ was shown in \cite{Ka88} and \cite{PZ06}, in \cite{PZ06}
regularity of the attractor of the semigroup was also established.
\par 
The global unique solvability, dissiptivity and asymptotic regularity of solutions of  \eqref{0.1g} {\it without} any growth restrictions 
 (just assuming that $g(u)$ satisfies 
\eqref{0.2} with arbitrary $q\in \R_+$)
 has been relatively recently established  in \cite{KaZ08} also for the case of solutions
belonging to the natural energy space
\begin{equation}\label{0.enspace}
\Cal E_{f=0}:=[H^1_0(\Omega)\cap L^{q+2}(\Omega)]\times L^2(\Omega).
\end{equation}
Thus, despite the expectations, even on the level of energy solutions there
is no critical exponent for the growth rate of $g$ and  the analytic and dynamic properties (existence and
uniqueness, dissipativity, asymptotic smoothing, attractors and
their dimension) look very similar for the cases $q<4$ and $q>4$.
This is related with the non-trivial monotonicity properties of the
equation considered in the space $L^2(\Omega)\times H^{-1}(\Omega)$,
see \cite{KaZ08} for more details.
\par
The alternative case when another nonlinearity vanishes $g\equiv0$ also
leads to essential simplifications.
Indeed, assuming that $h=0$ for simplicity   and introducing the new variable
$v(t):=\int_0^tu(s)\,ds$, we reduce \eqref{0.1} to
\begin{equation*}
\Dt^2v-\gamma\Dt v+F(\Dt v)-\Dx v=c,
\end{equation*}
where $c$ depends on the initial data and $F(u):=\int_0^u f(v)\,dv$.
Using e.g., the methods of \cite{Khan08}, one can show the absence of a critical
exponent for the growth rate of $f(u)$  in the energy  phase space.
\par
The situation becomes more complicated when both nonlinearities $f$ and $g$ are
presented in the equation and grow sufficiently fast since the methods developed to
treat the case of fast growing $g$ are hardly compatible with the methods for $f$
and vise versa. In particular, the problem of presence or absence of critical growth
exponents for the non-linearities $f$ and $g$ is still open here and, to the best of
our knowledge, a more or less complete theory for this equation (including existence
and uniqueness, dissipativity, asymptotic regularity, attractors, etc.) is built up only
for the case where $f$ and $g$ satisfy the growth restrictions $p\le 4$, $q\le4$
and additional monotonicity restriction
\begin{equation}\label{0.grad}
f(u)\ge0,\  \ u\in\R,
\end{equation}
see \cite{Khan2011} for the details.
\par
The main aim of these notes is to show that problem \eqref{0.1} is globally well-posed and dissipative at least in the class of the so-called strong solutions $\xi_u\in\Cal E_1$ {\it without} any restrictions on the growth exponents $p$ and $q$ and without the monotonicity assumption \eqref{0.grad}. To be more precise, the main result of the notes is the following theorem.

\begin{theorem}\label{Th0.main} Let $h\in L^2(\Omega)$ and the nonlinearities $f$ and $g$
satisfy assumptions \eqref{0.2}. Then, for every $\xi_u(0)\in \Cal E_1$,
there exists a unique strong solution $\xi_u\in C(\R_+,\Cal E_1)$ of \eqref{0.1}
and the following estimate holds:
\begin{equation}\label{0.4}
\|\xi_u(t)\|_{\Cal E_1}\le Q(\|\xi_u(0)\|_{\Cal E_1})e^{-\alpha t}+Q(\|h\|_{L^2})
\end{equation}
for some positive constant $\alpha$ and monotone function $Q$, where
$$
\|\xi_u(t)\|^2_{\Cal E_1}:=\|\Nx \Dt u(t)\|^2_{L^2}+\|\Dx u(t)\|^2_{L^2}.
$$
\end{theorem}
The proof of this theorem is given in Section \ref{s1}.
\par
The dissipative estimate \eqref{0.4} is strong enough to obtain the existence of a global
attractor $\Cal A$ for the considered system in the phase space $\Cal E_1$ and verify that
its smoothness is restricted by the regularity of the data $f$, $g$ and $h$ only, see
Section \ref{s2} for more details. Note also that, in contrast to the most part of papers
on the subject, we do not use the monotonicity assumption \eqref{0.grad}. As a result, the
equation does not possess any more a global Lyapunov function and the non-trivial dynamics
on the attractor becomes possible.
 For instance, our assumptions include the Van der Pole nonlinearities $f(u)=u^3-u$ and
 $g(u)=u$, so the time periodic orbits (and  chaotic dynamics) become possible.
 Another classical example with non-trivial dynamics is the so-called FitzHugh-Nagumo system:
\begin{equation}
\begin{cases} \Dt u=\Dx u-\phi(u)-v,\\
\Dt v=u-v.
\end{cases}
\end{equation}
Indeed, differentiating the first equation by $t$ and removing the variable $v$
 using the second equation, we obtain the equation
 \begin{equation}
 \Dt^2 u -\Dt\Dx u+\psi'(u)\Dt u-\Dx u+\psi(u)=0,
 \end{equation}
where $\psi (u)=u+ \phi(u)$. So, the FitzHugh-Nagumo
system is indeed a particular case of the  strongly damped wave equation of the form \eqref{0.1} .
\par
Thus, relaxing the monotonicity assumption \eqref{0.grad} indeed makes the
theory essentially more general and interesting.
  As a price to pay, we lose the control over the possible growth of weak energy solutions.
  Indeed, the energy equality for our problem reads
\begin{equation}\label{0.energy-eq}
\frac d{dt}\(\frac 12\|\Dt u\|^2_{L^2}+\frac12\|\Nx u\|^2_{L^2}+(G(u),1)-(h,u)\)+\gamma\|\Dt\Nx u\|^2_{L^2}+(f(u)\Dt u,\Dt u)=0
\end{equation}
(here and below $(u,v)$ stands for the classical inner product in $L^2(\Omega)$ and
$G(u):=\int_0^u g(v)\,dv$). We see that under the assumptions \eqref{0.2},
we have only the control $(f(u)\Dt u,\Dt u)\ge -C\|\Dt u\|^2_{L^2}$ which is enough to
prove the existence of weak energy  solutions but is {\it not sufficient} to verify that
they are globally bounded in time. Actually, we do not know how to obtain the dissipative
energy estimate on the level of weak energy solutions and by this reason have to consider
more smooth solutions $\xi_u\in \Cal E_1$.

\section{Main estimate}\label{s1}

The main aim of this section is to prove the key estimate \eqref{0.4}.
We start with slightly weaker dissipative estimate in the  space
$H^2(\Omega)\cap H^1_0(\Omega)\times L^2(\Omega)$.
In what follows to simplify notations we will denote by $C$ various constants that do not
depend on the initial data.
\begin{theorem}\label{Th1.main} Let the  conditions of the Theorem \ref{Th0.main} be satisfied
 and let $\xi_u\in C(\R_+,\Cal E_1)$ be a strong solution of \eqref{0.1}.
 Then the following estimate holds:
\begin{equation}\label{1.4}
\|\Dt u(t)\|_{L^2}^2+\|u(t)\|_{H^2}^2+\int_t^{t+1}\|\Dt u(s)\|^2_{H^1}\,ds\le Q(\|\Dt u(0)\|_{L^2}^2+\|u(0)\|_{H^2}^2)e^{-\alpha t}+Q(\|h\|_{L^2}),
\end{equation}
where the positive constant $\alpha$ and monotone function $Q$ which are independent of $\xi_u$.
\end{theorem}
\begin{proof} The proof of this estimate is strongly based on  the new estimate obtained by multiplication of \eqref{0.1} by $v:=\Dt u-\gamma\Dx u+F(u)$, where $F(u):=\int_0^uf(v)\,dv$. It is not difficult to show that under the above assumptions on the solution $\xi_u$, $v\in L^2(\Omega)$ and the multiplication is allowed.
 Then, after the straightforward transformations, we get
\begin{multline}\label{1.5}
\frac d{dt}\(\frac12\|v\|^2_{L^2}+\frac12\|\Nx u\|^2_{L^2}+(G(u),1)\)+\\+\gamma\|\Dx u\|^2_{L^2}+(f(u)+\gamma g'(u),|\Nx u|^2)+(F(u),g(u))=(h,v).
\end{multline}
In addition, due to \eqref{0.2} and the assumption $p+q>0$,
\begin{multline}
(f(u)+g'(u),|\Nx u|^2)+(F(u),g(u))\ge\\\ge
(|f(u)|+|g'(u)|,|\Nx u|^2)+\frac12(|F(u)|+1,|g(u)|+1)- C(\|\Nx u\|^2 +\|u\|^2+1)
\end{multline}
and using the interpolation $\|\Nx u\|^2_{L^2}\le\|\Dx u\|_{L^2}\|u\|_{L^2}$, we have
\begin{multline}\label{1.6}
\frac d{dt}\(\|v\|^2_{L^2}+\|\Nx u\|^2_{L^2}+2(G(u),1)\)+\\+\gamma\|\Dx u\|^2_{L^2}+(|f(u)|+\gamma|g'(u)|,|\Nx u|^2)+(|F(u)|+1,|g(u)|+1)\le C+2\|h\|_{L^2}\|v\|_{L^2}.
\end{multline}
This estimate is still not enough to get the desired dissipative estimate since we do not have the positive term related
with $\|v\|_{L^2}$ without the differentiation.
\par
At the next step, we use the energy equality \eqref{0.energy-eq} which is obtained by multiplication of \eqref{0.1} by $\Dt u$ and which together with our assumptions on $f$ gives
\begin{multline}\label{1.7}
\frac{d}{dt}\(\frac12\|\Dt u\|^2_{L^2}+\frac12\|\Nx u\|^2_{L^2}+(G(u),1)\)+\\+(|f(u)|+1,|\Dt u|^2)+\gamma\|\Dt\Nx u\|^2_{L^2}\le L(\|\Dt u\|^2_{L^2}+\|h\|^2_{L^2})
\end{multline}
for some positive constant $L$.
\par
To estimate the term in the right-hand side of \eqref{1.7}, we multiply equation \eqref{0.1} by $u$ which gives
\begin{equation}\label{1.8}
\|\Dt u\|^2_{L^2}=\frac d{dt}\((u,\Dt u)+\frac\gamma2\|\Nx u\|^2_{L^2}\)+\|\Nx u\|^2_{L^2}+(g(u),u)+(f(u)\Dt u,u)-(h,u).
\end{equation}
Note that, due to our assumptions \eqref{0.2} on the nonlinearity $f$, for any $\beta>0$,
\begin{equation}
|(f(u)\Dt u,u)|\le \beta(|f(u)|,|\Dt u|^2)+C_\beta(|f(u)|+1,u^2+1)
\end{equation}
Using now the  assumption that $p+q>0$, we see that $|f(u)|u^2\sim |u|^{p+2}$  and  $|F(u)g(u)|\sim |u|^{p+2+q}$ as $u\to\infty$, therefore
\begin{equation}\label{1.good}
|(f(u)\Dt u,u)|\le \beta(|f(u)|,|\Dt u|^2)+C_\beta(|F(u)|+1,|g(u)|+1)+C_\beta.
\end{equation}
Inserting this into the right-hand side of \eqref{1.7} and fixing $\beta>0$ being small enough,
we arrive at
\begin{multline}\label{1.9}
\frac{d}{dt}\(\frac12\|\Dt u\|^2_{L^2}+\frac12\|\Nx u\|^2_{L^2}+(G(u),1)-
L((u,\Dt u)+\frac\gamma2\|\Nx u\|^2_{L^2})\)+\\+\frac12(|f(u)|+1,|\Dt u|^2)+\gamma\|\Dt\Nx u\|^2_{L^2}\le
CL(|F(u)|+1,|g(u)|+1)+C(\|\Nx u\|^2_{L^2}+\|h\|^2_{L^2}+1).
\end{multline}
Taking a sum of \eqref{1.9} multiplied by a small parameter $\kappa>0$ with \eqref{1.6}, we have
\begin{multline}\label{1.10}
\frac d{dt}\(\|v\|^2_{L^2}+(1+\frac\kappa2-L\frac{\kappa\gamma}2)\|\Nx u\|^2_{L^2}+\right.\\\left.+(2+\kappa)(G(u),1)-L\kappa(u,\Dt u)+\frac\kappa2\|\Dt u\|^2_{L^2}\)+\\+
\beta\(\|\Dt u\|^2_{L^2}+\|\Dx u\|^2_{L^2}+(|f(u)|+1,|\Dt u|^2)+\right.\\\left.+(|f(u)|+\gamma|g'(u)|,|\Nx u|^2)+(|F(u)|+1,|g(u)|+1)+\|\Dt\Nx u\|^2_{L^2}\)\le\\\le C(\|h\|_{L^2}\|v\|_{L^2}+\|h\|^2_{L^2}+1).
\end{multline}
for some positive constant $\beta$ depending on $\kappa$. We now note that it is possible to fix $\kappa$ being small enough that the function
\begin{equation}
\Cal E_u(t):=\|v\|^2_{L^2}+(1+\frac\kappa2-L\frac{\kappa\gamma}2)\|\Nx u\|^2_{L^2}+(2+\kappa)(G(u),1)-L\kappa(u,\Dt u)+\frac\kappa2\|\Dt u\|^2_{L^2}
\end{equation}
will satisfy the inequalities
\begin{multline}\label{1.11}
\alpha\(\|v\|^2_{L^2}+\|\Dt u\|^2_{L^2}+\|\Nx u\|^2_{L^2}+(|G(u)|,1)\)-C_1\le
 \Cal E_u(t)\le\\\le C\(1+\|v\|^2_{L^2}+\|\Dt u\|^2_{L^2}+\|\Nx u\|^2_{L^2}+(|G(u)|,1)\)
\end{multline}
for some positive $\alpha$. Thus, \eqref{1.10} reads
\begin{multline}\label{1.12}
\frac d{dt}\Cal E_u(t)+
\beta\(\|\Dt u\|^2_{L^2}+\|\Dx u\|^2_{L^2}+(|f(u)|+1,|\Nx u|^2)+(|F(u)|+1,|g(u)|+1)\)+\\+\beta(|f(u)|+1,|\Dt u|^2)\le C(\|h\|_{L^2}\|v\|_{L^2}+\|h\|^2_{L^2}+1).
\end{multline}
Let now $q\ge p$. Then $$(|F(u)|+1)(|g(u)|+1)\sim|u|^{p+q+2}\ge|u|^{2p+2}\sim F(u)^2, $$
and using the obvious estimate
\begin{equation}
\|v\|^2_{L^2}\le C(\|\Dt u\|^2_{L^2}+\|\Dx u\|^2_{L^2}+\|F(u)\|^2_{L^2}),
\end{equation}
we see that, in the case $q\ge p$, \eqref{1.12} implies
\begin{equation}\label{1.13}
\frac d{dt}\Cal E_u(t)+\beta \Cal E_u(t)\le C(\|h\|^2_{L^2}+1)
\end{equation}
for some positive $\beta$.
\par
It only remains to study the case $p>q$. In this case, we extract the desired $L^2$ norm of
$F(u)$ from the term $(|f(u)|+1,|\Nx u|^2)$. Indeed
\begin{equation}\label{26}
(|f(u)|+1,|\Nx u|^2)\ge \alpha(|u|^p,|\Nx u|^2)=\alpha_1\|\Nx(|u|^{(p+2)/2})\|^2_{L^2}.
\end{equation}
Since $H^1_0(\Omega)$ is continuously embedded into $L^4(\Omega)$ we have
$$
\alpha_1\|\Nx(|u|^{(p+2)/2})\|^2_{L^2}\geq
\alpha_3\||u|^{p+2}\|_{L^2},
$$
 and we obtain from \eqref{26} that
$$
(|f(u)|+1,|\Nx u|^2)\ge
 \alpha_3\|F(u)\|_{L^2}^{\frac{p+2}{p+1}}-C
$$
for some positive $\alpha_i, \ i=1,2,3$.
 Thus, \eqref{1.12} now reads
\begin{equation}\label{1.14}
\Dt \Cal E_u(t)+\beta[\Cal E_u(t)]^{\frac{p+2}{2(p+1)}}\le
C(\|h\|_{L^2}[\Cal E_u(t)]^{1/2}+\|h\|^2_{L^2}+1).
\end{equation}
Since $\frac{p+2}{2(p+1)}>\frac12$, the Gronwall type inequality works in both cases
and gives the dissipative estimate for $\Cal E_u(t)$:
\begin{equation}\label{1.edis}
\Cal E_u(t)\le Q(\Cal E_u(0))e^{-\alpha t}+Q(\|h\|_{L^2}),
\end{equation}
where the positive constant $\alpha$ and monotone function $Q$ are independent of $t$
and $u$. Note also that, due to the maximal regularity result for the semilinear
heat equation,
\begin{equation}\label{1.vu}
C(\|\Dt u(t)\|^2_{L^2}+\|u(t)\|^2_{H^2})-C\le\|v(t)\|^2_{L^2}\le Q(\|\Dt u(t)\|^2_{L^2}+
\|u(t)\|^2_{H^2})
\end{equation}
for some positive $C$ and $Q$. The desired estimate \eqref{0.4}
follows in a straightforward way from \eqref{1.edis} and \eqref{1.vu}.
Thus, Theorem \ref{Th1.main} is proved.
\end{proof}
Next proposition gives the uniqueness of the strong solution of equation \eqref{0.1}.

\begin{proposition}\label{Prop1.unique} Let the conditions of the Theorem \ref{Th0.main}
 hold and let
$\xi_{u_1},\xi_{u_2}\in C(\R_+,\Cal E_1)$ be two solutions of the problem \eqref{0.1}.
Then, the following estimate holds:
\begin{equation}\label{1.lip}
\|\xi_{u_1}(t)-\xi_{u_2}(t)\|_{\Cal E}\le Ce^{Kt}\|\xi_{u_1}(0)-\xi_{u_2}(0)\|_{\Cal E},
\end{equation}
where the constants $C$ and $K$ depend on the initial data and $\|\xi_u\|^2_{\Cal E}:=\|\Nx u\|^2_{L^2}+\|\Dt u\|^2_{L^2}$.
\end{proposition}
\begin{proof}
Indeed, let $v:=u_1-u_2$. Then, this function solves
\begin{equation}\label{1.dif}
\Dt^2 v+f(u_1)\Dt v-\gamma\Dx\Dt v-\Dx v=-[f(u_1)-f(u_2)]\Dt u_2-[g(u_1)-g(u_2)].
\end{equation}
Multiplying this equation by $\Dt v$ and using the estimate \eqref{1.4}
together with the embeddings $H^2(\Omega)\subset C(\Omega)$ and
$H^1(\Omega)\subset L^4(\Omega)$, we have
\begin{multline}
\frac12\frac d{dt}\|\xi_v\|^2+\gamma\|\Nx\Dt v\|^2_{L^2}+(f(u_1)\Dt v,\Dt v)=-\\
-([f(u_1)-f(u_2)]\Dt u_2,\Dt v)-(g(u_1)-g(u_2),\Dt v)\le\\\le C\|v\|_{L^2}
\|\Dt v\|_{L^4}\|\Dt u_1\|_{L^4}+C\|v\|_{L^2}\|\Dt v\|_{L^2}\le\\\le \frac
\gamma2\|\Nx \Dt v\|^2_{L^2}+\|\Nx\Dt u_1\|^2_{L^2}\|v\|^2_{L^2}.
\end{multline}
Thus, we end up with the following inequality
\begin{equation}
\frac d{dt}\|\xi_v\|^2+\gamma\|\Nx\Dt v\|^2_{L^2}\le C\|\Nx \Dt u_2\|^2_{L^2}
\|\xi_v\|^2_{L^2}
\end{equation}
and the Gronwall inequality applied to this relation finishes the proof of the proposition.
\end{proof}
We are now ready to check the dissipativity in $\Cal E_1$.
\begin{proposition}\label{Prop1.e1dis} Let the conditions of the Theorem \ref{Th0.main}
be satisfied. Then, for every $\xi_u(0)\in \Cal E_1$, there is a unique
solution $\xi_u\in C(\R_+,\Cal E_1)$ of the problem \eqref{0.1} and the
following estimate holds:
\begin{equation}\label{1.e1dis}
\|\xi_u(t)\|_{\Cal E_1}+\int_t^{t+1}\|\Dt u(s)\|^2_{H^2}\,ds\le
Q(\|\xi_u(0)\|_{\Cal E_1})e^{-\alpha t}+Q(\|h\|_{L^2}),
\end{equation}
for some positive constant $\alpha$ and monotone function $Q$.
\end{proposition}
\begin{proof} We restrict ourselves to the formal derivation of the dissipative estimate
\eqref{1.e1dis}. The existence of a solution as well as the justification of this
derivation can be done in a standard way using, e.g., Galerkin approximations.
Moreover, due to \eqref{1.4}, we only need to obtain the control over the higher norms
 of the derivative $\Dt u$. To this end, we multiply equation \eqref{0.1} by $-\Dt\Dx u$.
 Then, after some transformations, we get
\begin{multline}\label{1.dt}
\frac12\frac d{dt}\|\xi_u\|^2_{\Cal E_1}+\|\xi_u\|^2_{\Cal E_1}+\gamma\|\Dt\Dx u\|^2_{L^2}=
(f(u)\Dt u,\Dt\Dx u)+\\+(g(u),\Dt \Dx u)+\|\Dt\Nx u\|^2_{L^2}+\|\Dx u\|^2_{L^2}\le
C(\|\Dt u\|^2_{L^2}+\|\Dx u\|^2_{L^2})+\frac\gamma2\|\Dt\Dx u\|^2_{L^2},
\end{multline}
where we have implicitly used that $H^2(\Omega)\subset C(\Omega)$ and the interpolation
$\|v\|_{H^1}^2\le C\|v\|_{L^2}\|v\|_{H^2}$.
The obtained estimate gives
\begin{equation}
\frac d{dt}\|\xi_u\|^2_{\Cal E_1}+\|\xi_u\|^2_{\Cal E_1}+
\gamma\|\Dt \Dx u\|^2_{L^2}\le C(\|\Dt u\|^2_{L^2}+\|\Dx u\|^2_{L^2})
\end{equation}
and the Gronwall inequality together with \eqref{1.4} finishes the proof of the proposition.
\end{proof}

\section{A global attractor}\label{s2}
In this section, we study the long-time behavior of solutions of the problem \eqref{0.1}
in terms of the associated global attractor.
 For the reader convenience, we first remind the key definitions of the attractors theory,
 see \cite{BV,tem} for more details.
 \par
 According to Proposition \ref{Prop1.e1dis}, the solution operators of the problem \eqref{0.1}
 generate a semigroup in the phase space $\Cal E_1$
\begin{equation}\label{2.sem}
S(t)\xi_u(0):=\xi_u(t),\ \ S(t):\Cal E_1\to\Cal E_1,\ \ S(t+h)=S(t)\circ S(h),\ t,h\ge0.
\end{equation}
Moreover, according to the estimate \eqref{1.e1dis}, the semigroup
$S(t)$ is {\it dissipative} in the phase space $\Cal E_1$, i.e., the
estimate
\begin{equation}\label{2.dis}
\|S(t)\xi\|_{\Cal E_1}\le Q(\|\xi\|_{\Cal E_1})e^{-\alpha t}+Q(\|h\|_{L^2}),\ \
\xi\in\Cal E_1
\end{equation}
holds for some positive constant $\alpha$ and monotone function $Q$.

\begin{definition}\label{Def2.set} Let $S(t):\Cal E_1\to\Cal E_1$ be a semigroup.
A set $\Cal B\subset\Cal E_1$ is called an attracting set for this semigroup if
for every bounded
set $B\subset\Cal E_1$ and every neighborhood $\Cal O(\Cal B)$ of the set
$\Cal B$, there exists $T=T(B,\Cal O)$ such that
\begin{equation}
S(t)B\subset \Cal O(\Cal B)
\end{equation}
for all $t\ge T$.
\end{definition}

\begin{definition}\label{Def2.attr} Let $S(t):\Cal E_1\to\Cal E_1$ be a semigroup.
A set $\Cal A$ is called a global attractor for the semigroup $S(t)$
if
\par
1. The set $\Cal A$ is compact in $\Cal E_1$;
\par
2. The set $\Cal A$ is strictly invariant: $S(t)\Cal A=\Cal A$ for all $t\ge0$;
\par
3. The set $\Cal A$ is an attracting set for the semigroup $S(t)$.
\end{definition}
To verify the existence of a global attractor, we will use the following version of an
abstract attractor existence theorem.

\begin{proposition}\label{Prop2.attr} Let $S(t):\Cal E_1\to\Cal E_1$
be a semigroup satisfying the following two assumptions:
\par
1.\ There exists a compact attracting set $\Cal B$ for the semigroup $S(t)$;
\par
2. \ For every $t\ge0$, the map $S(t):\Cal E_1\to\Cal E_1$  has a closed graph in $\Cal E_1\times\Cal E_1$.
\par
Then, this semigroup possesses a global attractor $\Cal A\subset\Cal B$ which is generated by all complete bounded trajectories:
\begin{equation}
\Cal A=\Cal K\big|_{t=0},
\end{equation}
where $\Cal K\subset L^\infty(\R,\Cal E_1)$ is a set of functions $u:\R\to\Cal E_1$
such that $S(t)u(h)=u(t+h)$ for all $h\in\R$ and~$t\ge0$.
\end{proposition}
For the proof of this proposition, see  \cite{PZ}.
\par
We are now ready to state and prove the main result of this section.
\begin{theorem}\label{Th2.main} Suppose that the conditions of the Theorem \ref{Th0.main}
 are satisfied. Then the  semigroup $S(t)$ associated with problem \eqref{0.1}
possesses a global attractor $\Cal A$ in the phase space $\Cal E_1$.
\end{theorem}
\begin{proof} Indeed, the second assumption of Proposition \ref{Prop2.attr} is an
immediate corollary of Proposition \ref{Prop1.unique}, so we only need
to check the first one. To this end, we split a solution $u$ of equation \eqref{0.1}
in a sum $u(t):=v(t)+w(t)$, where the
 function $v$ solves the linear problem:
 \begin{equation}\label{2.linear}
 \begin{cases}
 \Dt^2v-\gamma\Dt\Dx v-\Dx v=h,\ \ v\big|_{\partial\Omega}=0,\\
\xi_v\big|_{t=0}=\xi_u\big|_{t=0}
\end{cases}
 \end{equation}
 and the reminder $w$ satisfies
 \begin{equation}\label{2.w}
 \begin{cases}
 \Dt^2w-\gamma\Dt\Dx w-\Dx w=-f(u)\Dt u-g(u),\ \ \ \ w\big|_{\partial\Omega}=0,\\
\xi_w\big|_{t=0}=0.
\end{cases}
 \end{equation}
 Moreover, without loss of generality, we may assume that $g(0)=0$.
  The properties of functions $v$ and $w$
  are collected in the following two lemmas.
 \begin{lemma}\label{Lem2.conv} Let the above assumptions hold and let
 $H=H(x)\in H^2(\Omega)\cap H^1_0(\Omega)$ be the solution of the
 problem
 $$
 -\Dx H=h, \ x \in \Omega; \ u\big|_{\partial\Omega}=0.
$$
 Then, the following estimate
holds:
 \begin{equation}\label{2.decay}
 \|\xi_v(t)-\xi_H\|_{\Cal E_1}\le C(\|\xi_u(0)\|_{\Cal E_1}+\|h\|_{L^2})e^{-\alpha t},
 \end{equation}
 where $\xi_H=(H,0)$ and positive constants $C$ and $\alpha$ are independent of $u$.
 \end{lemma}
 \begin{proof}[Proof of the Lemma] Indeed, introducing the new variable $\tilde v(t):=v(t)-H$, we reduce \eqref{2.linear} to the homogeneous form
 \begin{equation}
 \begin{cases}
 \Dt^2 \tilde v-\gamma\Dt\Dx \tilde v-\Dx \tilde v=0,\ \ \tilde v\big|_{\partial\Omega}=0,\\
\xi_{\tilde v}\big|_{t=0}=\xi_u\big|_{t=0}-\xi_H.
\end{cases}
 \end{equation}
 Multiplying this equation by $\Dt\Dx\tilde v+\beta\Dx\tilde v$, where $\beta$ is a small
 positive parameter, and arguing in a standard way
 we derive that
 \begin{equation}
 \|\xi_{\tilde v}(t)\|^2_{\Cal E_1}\le C\|\xi_{\tilde v}(0)\|^2_{\Cal E_1}e^{-\alpha t}
 \end{equation}
 for some positive $C$ and $\alpha$, see e.g., \cite{BV,tem} as well as the proof of
 Lemma \ref{Lem2.comp} below. The desired estimate \eqref{2.decay} is an immediate
 corollary of this estimate and Lemma \ref{Lem2.conv} is proved.
 \end{proof}
Thus, we have proved that the $v$ component of the solution $u$ converges
exponentially to a single function $H\in H^2(\Omega)$ which is independent of time
and the initial data. The next lemma shows that the $w$ component is more regular.
\begin{lemma}\label{Lem2.comp} Let the above assumptions hold and let
\begin{equation}
\Cal E_2:=[H^3(\Omega)\cap\{u\big|_{\partial\Omega}=\Dx u\big|_{\partial\Omega}=0\}]
\times [H^2(\Omega)\cap H^1_0(\Omega)].
\end{equation}
Then the solution $w$ of problem \eqref{2.w} belongs to $\Cal E_2$ for all $t\ge0$ and
the following estimate holds:
\begin{equation}\label{2.wcomp}
\|\xi_w(t)\|_{\Cal E_2}\le Q(\|\xi_u(0)\|_{\Cal E_1})e^{-\alpha t}+Q(\|h\|_{L^2}),
\end{equation}
for some positive constant $\alpha$ and monotone function $Q$ which are independent of $u$.
\end{lemma}
\begin{proof}[Proof of the Lemma] We give below only the formal derivation of estimate
\eqref{2.wcomp} which can be justified e.g., using the Galerkin approximations.
First, due to the assumption $g(0)=0$, it follows from the equation \eqref{2.w} that
at least formally $\Dx w\big|_{\partial\Omega}=0$, so we may multiply equation
\eqref{2.w} by $\Dt\Dx^2 w+\beta \Dx^2 w$ and do integration by parts. This gives
\begin{multline}\label{2.huge}
\frac d{dt}\(\frac 12\|\xi_w\|^2_{\Cal E_2}+\beta(\Nx\Dx w,\Nx \Dt w)+
\frac{\gamma\beta}2\|\Nx\Dx w\|^2_{L^2}\)+\beta\|\Nx\Dx w\|^2_{L^2}-
\beta\|\Dt\Dx w\|^2_{L^2}+\\+\gamma\|\Dt\Nx\Dx w\|^2_{L^2}=(\Nx(f(u)\Dt u+g(u)),
\Nx(\Dt\Dx w+\beta\Dx w)).
\end{multline}
Fixing $\beta>0$ small enough and using the notation
$$E_2(w):=\frac 12\|\xi_w\|^2_{\Cal E_2}+\beta(\Nx\Dx w,\Nx \Dt w)+
\frac{\gamma\beta}2\|\Nx\Dx w\|^2_{L^2}$$ we see that, on the one hand,
\begin{equation}\label{2.equiv}
C_1\|\xi_w\|_{\Cal E_2}^2\le E_2(w)\le C_2\|\xi_w\|^2_{\Cal E_2}
\end{equation}
for some positive constants $C_1$ and $C_2$. On the other hand, the
equation \eqref{2.huge} implies that
\begin{equation}\label{2.gron}
\frac d{dt}E_2(w)+\alpha E_2(w)\le C(\|f(u)\Dt u\|^2_{H^1}+\|g(u)\|^2_{H^1})
\end{equation}
for some positive constants $C$ and $\alpha$. Finally, using the embedding
$H^2(\Omega)\subset C(\Omega)$ and growth restrictions \eqref{0.2},
we estimate the right-hand side of \eqref{2.gron} as follows:
\begin{equation}\label{2.last}
\|f(u)\Dt u\|^2_{H^1}+\|g(u)\|^2_{H^1}\le C(\|u\|_{H^2}^{p+q+1}+1)(\|\Dt\Nx u\|^2_{L^2}+1).
\end{equation}
Applying the Gronwall inequality to \eqref{2.gron} and using
\eqref{2.last} and \eqref{2.equiv} together with the dissipative estimate \eqref{1.e1dis},
we derive the desired estimate \eqref{2.wcomp} and finish the proof of
Lemma \ref{Lem2.comp}.
\end{proof}
It is not difficult now to finish the proof of the theorem. Indeed,
 Lemmas \ref{Lem2.conv} and \ref{Lem2.comp} show that the set
\begin{equation}
\Cal B:=\xi_H+\{w\in \Cal E_2,\ \|\xi_w\|_{\Cal E_2}\le R\}
\end{equation}
will be a compact attracting set for the semigroup $S(t)$ generated
by the problem \eqref{0.1} if $R$ is large enough. Thus, all
assumptions of the Proposition \ref{Prop2.attr} are verified and the
theorem is proved.
\end{proof}

\begin{remark} As already it was mentioned in the introduction, we do not know how to
deduce the basic dissipative estimate for the {\it weak} solutions
of the  problem \eqref{0.1} in the phase space $\Cal E$ in the case when the condition
\eqref{0.grad} is violated. However, as follows from the Theorem \ref{Th1.main},
we have such an estimate in the intermediate space
$[H^2(\Omega)\cap H^1(\Omega)]\times L^2(\Omega)$ which is in a sense natural
for strongly damped wave equations, see \cite{PZ06,KaZ08}.
 Actually, the problem is well posed in this space and the above developed
  attractor theory can be extended to to this phase space as well.
\end{remark}



\begin{thebibliography}{99}

\bibitem{BV}
A.  Babin and M.  Vishik. {\it Attractors of Evolution Equations}, Amsterdam: North-Holland,
1992.
\bibitem{CaCho02}
A. Carvalho and J. Cholewa. {\it Attractors for strongly damped wave equations with
critical nonlinearities}, Pacific J. Math., 207 (2002) 287--310.
\bibitem{CCD1}
A. Carvalho, J. Cholewa and T. Dlotko. {\it Strongly damped
wave problems: bootstrapping and regularity of solutions}. J. Diff.
Eqns., {\bf 244}, no. 9, (2008),  2310--2333.
\bibitem{DP}
 F. Dell'Oro and V. Pata. {\it Strongly damped wave equations with critical nonlinearities}. Nonlinear Anal.  75,  no. 14, (2012) 5723--5735.
\bibitem{GM}J. Ghidaglia and A.~Marzocchi. {\it Longtime behaviour of strongly damped wave equations,
global attractors and their dimension}, SIAM J. Math. Anal., {\bf 22}
(1991),  879--895.
\bibitem{Ka86}
V. K. Kalantarov. {\it Attractors for some nonlinear problems of
mathematical physics}, Zap. Nauchn. Sem. Leningrad. Otdel. Mat.
Inst. Steklov, (LOMI), 152 (1986) 50--54.

\bibitem{Ka88} V. K. Kalantarov. {\it  Global behavior of solutions of nonlinear
equations of mathematical physics of classical and non-classical
types}, Post Doct. Thesis, St. Petersburg Department of Steklov
Math. Inst., St. Petersburg, 1988.

\bibitem{KaZ08}
V. Kalantarov and S. Zelik. {\it Finite-dimensional attractors for the quasi-linear
strongly-
damped wave equation}, J. Diff. Equations, 247 (2009) 1120--1155.
\bibitem{Khan08}
A. Khanmamedov. {\it Long-time behaviour of wave equations with nonlinear interior damping,} Discrete Contin. Dyn. Syst.  21, no. 4, (2008) 1185--1198.
\bibitem{Khan2010}
A. Khanmamedov. {\it Global attractors for 2-D wave equations with displacement
dependent damping}, Math. Methods Appl. Sci., 33 (2010) 177--187.
\bibitem{Khan2011}
A. Khanmamedov. {\it Global attractors for strongly damped wave equations with displacement dependent damping and nonlinear source term of critical exponent}, Discrete Contin. Dyn. Syst.  31, no. 1,  (2011) 119--138.
\bibitem{Khan2012} A. Khanmamedov. {\it Strongly damped wave equation with exponential nonlinearities}, arXiv:1212.2180.
\bibitem{Ma} P.~Massatt. {\it Limiting behavior for strongly damped nonlinear wave
equations}, J. Differential Equations, {\bf 48} (1983), 334--349.
\bibitem{PS05}
V. Pata and M. Squassina. {\it On the strongly damped wave equation}, Commun. Math. Phys.,
253 (2005) 511--533.
\bibitem{PZ06}
V. Pata and S. Zelik. {\it Smooth attractors for strongly damped wave equations},
 Nonlinearity,
19 (2006) 1495--1506.
\bibitem{PZ06-1}
V. Pata and S. Zelik. {\it Global and exponential attractors for 3-D wave equations with displacement dependent damping}, Math. Methods Appl. Sci., 29 (2006) 1291--1306.
\bibitem{PZ}
V. Pata and S. Zelik. {\it A result on the existence of  global attractors
for semigroups of closed operators}, Comm. Pure Appl. Anal., vol. 6, no. 2, (2007), 481--486.
\bibitem{tem}
R. Temam. {\it Infinite-Dimensional Dynamical Systems in Mechanics and Physics}, New York:
Springer-Verlag, 1988.
\bibitem{YS09}
M. Yang and C. Sun. {\it Dynamics of strongly damped wave equations in
locally uniform spaces: attractors and asymptotic regularity},
Trans. Amer. Math. Soc., {\bf 361} (2009), no. 2, 1069--1101.
 \bibitem{We} G.~Webb. {\it Existence and asymptotic behavior for a strongly
damped nonlinear wave equation}, Canad. J. Math., {\bf 32}
(1980), 631--643.
\end{thebibliography}
\end{document}